\documentclass[12pt, a4paper]{article}
\usepackage{amsfonts}
\usepackage{mathrsfs}
\usepackage{latexsym}
\usepackage{xy}
\usepackage{amsfonts,amsmath,amssymb,amsthm}
\usepackage{color}

\xyoption{all}

\newcommand{\bcen}{\begin{center}}     \newcommand{\ecen}{\end{center}}
\newcommand{\bay}{\begin{array}}      \newcommand{\eay}{\end{array}}
\newcommand{\beq}{\begin{eqnarray*}}      \newcommand{\eeq}{\end{eqnarray*}}

\def\gl{\mathrm{gl.dim}}

\def\Hom{\mathrm{Hom}}

\def\op{\mathrm{op}}

\def\Ext{\mathrm{Ext}}

\def\End{\mathrm{End}}

\def\dim{\mathrm{dim}}

\def\Im{\mathrm{Im}}

\def\proj{\mathrm{proj}}
\def\Proj{\mathrm{Proj}}

\def\RHom{\mathrm{RHom}}
\def\rep{\mathrm{rep}}

\def\Tria{\mathrm{Tria}}
\def\Tor{\mathrm{Tor}}

\begin{document}

\newtheorem{theorem}{Theorem}
\newtheorem{proposition}{Proposition}
\newtheorem{lemma}{Lemma}
\newtheorem{corollary}{Corollary}
\newtheorem{remark}{Remark}
\newtheorem{example}{Example}
\newtheorem{definition}{Definition}
\newtheorem*{conjecture}{Conjecture}
\newtheorem{question}{Question}

\title{\large\bf Recollements and Hochschild theory}

\author{\large Yang Han}

\date{\footnotesize KLMM, ISS, AMSS,
Chinese Academy of Sciences, Beijing 100190, P.R. China.\\ E-mail:
hany@iss.ac.cn}

\maketitle

\bcen{\it Dedicated to the memory of Dieter Happel}\ecen

\begin{abstract} It is shown that a recollement of derived categories of algebras
induces those of tensor product algebras and opposite algebras
respectively, which is applied to clarify the relations between
recollements of derived categories of algebras and smoothness and
Hochschild cohomology of algebras.
\end{abstract}

\medskip


\medskip


\bigskip

\section{\large Introduction}

\indent\indent Let $\mathcal{T}_1$, $\mathcal{T}$ and
$\mathcal{T}_2$ be triangulated categories. A {\it recollement} of
$\mathcal{T}$ relative to $\mathcal{T}_1$ and $\mathcal{T}_2$ is
given by
$$\unitlength=1mm \begin{picture}(55,14) \put(0,5){$\mathcal{T}_1$}
\put(26,5){$\mathcal{T}$} \put(50,5){$\mathcal{T}_2$}
\put(24,10){\vector(-1,0){18}} \put(15,11){$i^*$}
\put(49,10){\vector(-1,0){18}} \put(40,11){$j_!$}
\put(6,5){\vector(1,0){18}} \put(10,6){$i_*=i_!$}
\put(31,5){\vector(1,0){18}} \put(35,6){$j^!=j^*$}
\put(24,0){\vector(-1,0){18}} \put(15,1){$i^!$}
\put(49,0){\vector(-1,0){18}} \put(40,1){$j_*$}
\end{picture}$$
and denoted by 9-tuple $(\mathcal{T}_1, \mathcal{T}, \mathcal{T}_2,
i^*, i_*=i_!, i^!, j_!, j^!=j^*, j_*)$ such that

(R1) $(i^*,i_*), (i_!,i^!), (j_!,j^!)$ and $(j^*,j_*)$ are adjoint
pairs of triangle functors;

(R2) $i_*$, $j_!$ and $j_*$ are full embeddings;

(R3) $j^!i_*=0$ (and thus also $i^!j_*=0$ and $i^*j_!=0$);

(R4) for each $X \in \mathcal {T}$, there are triangles

$$\begin{array}{l} j_!j^!X \rightarrow X  \rightarrow i_*i^*X  \rightarrow
\\ i_!i^!X \rightarrow X  \rightarrow j_*j^*X  \rightarrow
\end{array}$$ where the arrows to and from $X$ are the counit and the
unit, respectively.

Recollements of triangulated categories are ``short exact
sequences'' of triangulated categories. They were introduced by
Beilinson-Bernstein-Deligne \cite{BBD} and play an important role in
algebraic geometry \cite{BBD}, representation theory \cite{CPS,PS},
etc. Let $k$ be a field and $\otimes := \otimes_k$. Throughout the
paper, all algebras are assumed to be associative $k$-algebras with
identity, and all modules are right unitary modules unless stated
otherwise. Here, we focus on recollements of derived categories of
algebras, i.e., all triangulated categories in the recollements are
derived categories of algebras, which are closely related to tilting
theory \cite{AKL,K,M}, (co)localization theory \cite{M1,Kr2}, some
important homological invariants of algebras such as global
dimension \cite{W,K,AKLY}, finitistic dimension \cite{H}, Hochschild
homology and cyclic homology \cite{Ke2}, and so on.

In this paper, we shall show that a recollement of derived
categories of algebras induces those of tensor product algebras (see
section 3) and opposite algebras (see section 4) respectively. As
applications, we shall clarify the relations between recollements of
derived categories of algebras and smoothness, i.e., finiteness of
Hochschild dimension, and Hochschild cohomology of algebras. Note
that the relations between recollements of derived categories of
algebras and Hochschild homology and cyclic homology have been
clarified already in \cite{Ke2}. More precisely, we shall show in
section 5 that, in a recollement of derived categories of algebras,
the middle algebra is smooth if and only if so are the algebras on
both sides. As a corollary, a triangular matrix algebra is smooth if
and only if so are the algebras on diagonal. In section 6, we shall
obtain three triangles on Hochschild cocomplexes which can induce
three long exact sequences on Hochschild cohomologies of algebras.
Note that these long exact sequences on Hochschild cohomologies have
been widely studied for one-point extensions \cite{H1,GMS},
triangular matrix algebras \cite{C,MP,GS,CMRS,BG}, stratifying
ideals \cite{KN}, homological epimorphisms \cite{PX,S}, etc.

\section{Standard recollements}

\indent\indent In this section, we show that every recollement of
derived categories of algebras is equivalent to a standard one in
which all triangle functors are naturally isomorphic to derived
functors. This result is already known for some experts.

\subsection{Recollements of derived categories of algebras}

\indent\indent Let $A$ be an algebra. Denote by $\Proj A$ (resp.
$\proj A$) the category of projective (resp. finitely generated
projective) $A$-modules. Denote by $D(A)$ the unbounded derived
category of complexes of $A$-modules. Let $X$ be an object in
$D(A)$. Denote by $X^{\perp}$ the full subcategory of $D(A)$
consisting of all objects $Y \in D(A)$ such that
$\Hom_{D(A)}(X,Y[n]) =0$ for all $n \in \mathbb{Z}$. Denote by
$\mbox{Tria} X$ the smallest full triangulated subcategory of $D(A)$
which contains $X$ and is closed under small coproducts. We say $X$
is {\it exceptional} if $\Hom_{D(A)}(X , X[n]) = 0$ for all $n \in
\mathbb{Z}\backslash\{0\}$. We say $X$ is {\it compact} if the
functor $\Hom_{D(A)}(X , -)$ preserves small coproduct, or
equivalently, $X$ is {\it perfect}, i.e., isomorphic in $D(A)$ to an
object in $K^b(\proj A)$, the homotopy category of bounded complexes
of finitely generated projective $A$-modules. We say $X$ is {\it
self-compact} if $\Hom_{D(A)}(X , -)$ preserves small coproducts in
$\mbox{Tria} X$ (ref. \cite{J}).

A very important criterion for the right bounded derived category of
an algebra to admit a recollement is provided in \cite{K} (cf.
\cite[Theorem 3]{NS2}). It was extended and modified to suit for the
unbounded derived categories of algebras (ref. \cite[Corollary
3.4]{NS1}), differential graded algebras (ref. \cite[Theorem
3.3]{J}) and differential graded categories (ref. \cite[Corollary
3.4]{NS1}).

\begin{proposition} \label{criterion} {\rm (K\"{o}nig \cite{K};
J{\o}rgensen \cite{J}; Nicol\'{a}s-Saorin \cite{NS1})} Let $A_1, A$
and $A_2$ be algebras. Then $D(A)$ admits a recollement relative to
$D(A_1)$ and $D(A_2)$ if and only if there are objects $X_1$ and
$X_2$ in $D(A)$ such that

{\rm (1)} $\End_{D(A)}(X_i) \cong A_i$ as algebras for $i=1,2$;

{\rm (2)} $X_2$ (resp. $X_1$) is exceptional and compact (resp.
self-compact);

{\rm (3)} $X_1 \in X_2^{\perp}$;

{\rm (4)} $X_1^{\perp} \cap X_2^{\perp} = \{0\}$.

\end{proposition}

An important example of recollements of derived categories of
algebras is given by stratifying ideals:

\begin{example} \label{stratifying ideal} {\rm  (Cline-Parshall-Scott \cite{CPS2})
Let $A$ be an algebra, $e$ an idempotent of $A$, and $AeA$ a
stratifying ideal of $A$, i.e., the multiplication in $A$ induces an
isomorphism $Ae \otimes_{eAe} eA \cong AeA$ and $\Tor
^{eAe}_n(Ae,eA)=0$ for all $n \geq 1$. Then there is a recollement
$(D(A/AeA), D(A), D(eAe), i^*, i_* \linebreak =i_!, i^!, j_!,
j^!=j^*, j_*)$ where
$$\begin{array}{lll} i^* = - \otimes^L_A A/AeA, && j_! = - \otimes^L_{eAe}eA,\\
i_*=i_!= - \otimes^L_{A/AeA}A/AeA,
&& j^!=j^*= - \otimes^L_A Ae,\\
i^!=\RHom_A(A/AeA,-), && j_*=\RHom_{eAe}(Ae,-).
\end{array}$$
}\end{example}

\subsection{Standard recollements}

\begin{definition}{\rm Let $A_1,A$ and $A_2$ be algebras.
A recollement $(D(A_1), D(A), \linebreak D(A_2), i^*,i_*=i_!, i^!,
j_!, j^!=j^*, j_*)$ is said to be {\it standard} and {\it defined
by} $Y \in D(A^{\op} \otimes A_1)$ and $Y_2 \in D(A_2^{\op} \otimes
A)$ if $i^* \cong - \otimes^L_A Y$ and $j_! \cong -
\otimes^L_{A_2}Y_2$. }\end{definition}

\begin{proposition} \label{standard}
Let $A_1,A$ and $A_2$ be algebras, and $(D(A_1), D(A), D(A_2),
i^*,i_* \linebreak =i_!, i^!, j_!, j^!=j^*, j_*)$ a standard
recollement defined by $Y \in D(A^{\op} \otimes A_1)$ and $Y_2 \in
D(A_2^{\op} \otimes A)$. Then
$$\begin{array}{lll} i^*  \cong - \otimes^L_A Y , && j_! \cong - \otimes^L_{A_2}Y_2,\\
i_*=i_! \cong \RHom_{A_1}(Y,-), && j^!=j^* \cong \RHom_A(Y_2,-),\\
i^! \cong \RHom_A(\RHom_{A_1}(Y,A_1),-), && j_* \cong
\RHom_{A_2}(\RHom_A(Y_2,A),-).
\end{array}$$
\end{proposition}

\begin{proof} Since $i_*$ is a left adjoint of $i^!$,
it commutes with small coproduct. The functor $- \otimes^L_A Y \cong
i^*$ has a right adjoint $i_*$ which commutes with small coproduct,
thus $Y$ is compact in $D(A_1)$. Therefore, $\RHom_{A_1}(Y,-) \cong
- \otimes^L_{A_1} \RHom_{A_1}(Y,A_1)$ (ref. \cite[Lemma 2.6]{M}).
Since the right adjoint is unique up to natural isomorphism, we have
$i_*=i_! \cong \RHom_{A_1}(Y,-)$ and $i^! \cong
\RHom_A(\RHom_{A_1}(Y,A_1),-)$. Similar for $j^!=j^*$ and $j_*$.
\end{proof}

\begin{proposition} \label{standard} Let $A$, $A_1$ and $A_2$ be algebras. If $D(A)$ admits a
recollement relative to $D(A_1)$ and $D(A_2)$ then $D(A)$ admits a
standard recollement relative to $D(A_1)$ and $D(A_2)$.
\end{proposition}

\begin{proof} Let $(D(A_1), D(A), D(A_2), i^*,i_*=i_!, i^!, j_!, j^!=j^*, j_*)$ be
a recollement. It follows from Proposition~\ref{criterion} that
there are objects $X_i, i=1,2,$ in $D(A)$ such that they satisfy all
conditions in Proposition~\ref{criterion}. Clearly, we may assume
that $X_2$ is homotopically projective. Since $X_2$ is exceptional,
it follows from \cite[8.3.1]{KZ} that there exists $Y_2 \in
D(A_2^{\op} \otimes A)$ such that the derived tensor functor $-
\otimes^L_{A_2} Y_2 : D(A_2) \rightarrow D(A)$ sends $A_2$ to $X_2$.
By \cite[Theorem 2.8]{M}, we have a recollement $(X_2^{\perp}, D(A),
D(A_2), i'^*,i'_* =i'_!, i'^!, j'_!, j'^!=j'^*, j'_*)$ where $j'_! =
- \otimes^L_{A_2}Y_2, j'^!=j'^*=\RHom_A(Y_2,-),
j'_*=\RHom_{A_2}(\RHom_A(Y_2,A),-),$ and $i'_*=i'_!$ is the natural
embedding.

By \cite[Theorem 3]{BH} or \cite[\S 4 Theorem]{NS1}, there is a
homological epimorphism of differential graded algebras $f : A
\rightarrow C$ such that the essential image $\Im f_*$ of the
induced functor $f_* : D(C) \rightarrow D(A)$ and $\Im i_* =
X_2^{\perp}$ coincide. In particular, $D(A_1)$ and $D(C)$ are
equivalent as triangulated categories. It follows from \cite{Ke}
that there is a two-sided tilting differential graded $A_1$-$C$
bimodule $Z$ such that $-\otimes^L_{A_1}Z : D(A_1) \rightarrow D(C)$
is a triangle equivalence. Thus the composition $D(A_1) \stackrel{-
\otimes^L_{A_1}Z}{\rightarrow} D(C) \stackrel{f_*}{\rightarrow}
D(A)$ is a fully faithful triangle functor which is naturally
isomorphic to $- \otimes^L_{A_1} Y_1$, where $Y_1$ is the image of
$Z$ under the functor $f_* : D(A_1^{\op} \otimes C) \rightarrow
D(A_1^{\op} \otimes A)$. Since $_{A_1}Z$ is compact in
$D(A_1^{\op})$, so is $_{A_1}Y_1$. Thus $Y:=
\RHom_{A_1^{\op}}(Y_1,A_1)$ is a differential graded
$A$-$A_1$-bimodule, i.e., a complex of $A$-$A_1$-bimodules, which is
compact in $D(A_1)$ so there is a natural isomorphism
$\RHom_{A_1}(Y,-) \cong -\otimes^L_{A_1}Y_1 : D(A_1) \rightarrow
D(A)$. Consequently, $- \otimes^L_AY$ is left adjoint to $-
\otimes^L_{A_1}Y_1$. Thus $D(A)$ admits a standard recollement
relative to $D(A_1)$ and $D(A_2)$ defined by $Y \in D(A^{\op}
\otimes A_1)$ and $Y_2 \in D(A_2^{\op} \otimes A)$. \end{proof}

\begin{remark}{\rm Two recollements $(\mathcal{T}_1, \mathcal{T}, \mathcal{T}_2, i^*,
i_*=i_!, i^!, j_!, j^!=j^*, j_*)$ and $(\mathcal{T}'_1,
\mathcal{T}', \mathcal{T}'_2, i'^*, i'_*=i'_!, i'^!, j'_!,
j'^!=j'^*, j'_*)$ are said to be {\it equivalent} if $(\Im i_*, \Im
j_! , \linebreak \Im j_*) = (\Im i'_*, \Im j'_! , \Im j'_*)$. From
the proof of Proposition~\ref{standard}, it is easy to see that the
new constructed recollement is equivalent to the original given
one.}\end{remark}

\section{Recollements on tensor product algebras}

\indent\indent In this section, we show that from a standard
recollement of derived categories of algebras we can obtain those of
tensor product algebras.

\begin{lemma} \label{compact-selfcompact}  Let $A$ and $B$ be algebras, and $X,Y \in D(A)$. Then
the canonical homomorphism $B \otimes \Hom_{D(A)}(X,Y) \rightarrow
\Hom_{D(B \otimes A)}(B \otimes X, B \otimes Y)$ is an isomorphism
when

{\rm (1)} $X$ is compact, or

{\rm (2)} $X$ is self-compact and $Y \in \Tria X$.

\end{lemma}

\begin{proof} Since the functor $B \otimes - : D(A) \rightarrow D(B \otimes
A)$ is left adjoint to the forgetful functor, we have $\Hom_{D(B
\otimes A)}(B \otimes X, B \otimes Y) \cong \Hom_{D(A)}(X, B \otimes
Y)$, which is further isomorphic to $B \otimes \Hom_{D(A)}(X,Y)$
when $X$ is compact, or $X$ is self-compact and $Y \in \Tria X$.
\end{proof}

\begin{lemma} \label{delete B} Let $A, B$ and $C$ be algebras, and $Y \in D(A^{\op} \otimes C)$.
Then there are natural isomorphisms

{\rm (1)} $- \otimes^L_{B \otimes A}(B \otimes Y) \cong -
\otimes^L_A Y : D(B \otimes A) \rightarrow D(B \otimes C)$, and

{\rm (2)} $\RHom_{B \otimes C}(B \otimes Y,-) \cong \RHom_C(Y,-) :
D(B \otimes C) \rightarrow D(B \otimes A)$.

\end{lemma}

\begin{proof} (1) holds since $- \otimes^L_{B \otimes A}(B \otimes Y)
\cong (- \otimes^L_B B) \otimes^L_A Y \cong - \otimes^L_A Y$. (2)
holds since the functor $B \otimes - : D(A) \rightarrow D(B \otimes
A)$ is left adjoint to the forgetful functor.
\end{proof}

\begin{theorem} \label{tensor} Let $A,A_1,A_2$ and $B$ be algebras, and
$Y \in D(A^{\op} \otimes A_1)$ and $Y_2 \in D(A_2^{\op} \otimes A)$
define a standard recollement of $D(A)$ relative to $D(A_1)$ and
$D(A_2)$. Then $B \otimes Y$ and $B \otimes Y_2$ define a standard
recollement of $D(B \otimes A)$ relative to $D(B \otimes A_1)$ and
$D(B \otimes A_2)$. Moreover, the six triangle functors in this
recollement, now denoted with capital letters, are
$$\begin{array}{lll} I^*  \cong - \otimes^L_A Y , && J_! \cong - \otimes^L_{A_2}Y_2,\\
I_*=I_! \cong \RHom_{A_1}(Y,-), && J^!=J^* \cong \RHom_A(Y_2,-),\\
I^! \cong \RHom_A(\RHom_{A_1}(Y,A_1),-), && J_* \cong
\RHom_{A_2}(\RHom_A(Y_2,A),-).
\end{array}$$

\end{theorem}

\begin{proof} We may assume that $Y$ and $Y_2$ are homotopically
projective, and $Y_1 := \RHom_{A_1}(Y,A_1)$. Then the objects $X_i
:= (Y_i)_A, i=1,2,$ in $D(A)$ satisfy all conditions in
Proposition~\ref{criterion} (ref. \cite{J,K}). Let $Z_i := B \otimes
X_i$ for $i=1,2$. Now we show that $Z_1$ and $Z_2$ satisfy all
conditions in Proposition~\ref{criterion} for tensor product
algebras.

{\it Step 1.} Since $X_2$ is compact in $D(A)$, $Z_2$ is compact in
$D(B \otimes A)$. Since $X_2$ is compact and exceptional and
$\End_{D(A)}(X_2) \cong A_2$ as algebras, by
Lemma~\ref{compact-selfcompact}, we have
$$\Hom_{D(B \otimes A)}(Z_2,Z_2[n]) \cong B \otimes \Hom_{D(A)}(X_2,X_2[n])
\cong \left\{\begin{array}{ll} B \otimes A_2, & \mbox{\rm if } n=0; \\
0, & \mbox{\rm otherwise.} \end{array}\right. $$ Thus $Z_2$ is
exceptional and $\End_{D(B \otimes A)} (Z_2) \cong B \otimes A_2$ as
algebras.

{\it Step 2.} The forgetful functor $D(B \otimes A) \rightarrow
D(A)$ maps the objects in $\Tria_{D(B \otimes A)}(B \otimes X_1)$ to
the objects in $\Tria_{D(A)}(X_1)$. Indeed, let $\mathcal {C}$ be
the full triangulated subcategory of $D(B \otimes A)$ consisting of
the objects which become objects of $\Tria_{D(A)}(X_1)$ when
applying the forgetful functor. Then $\mathcal {C}$ is a full
triangulated subcategory of $D(B \otimes A)$ containing $B \otimes
X_1$ and closed under small coproducts. Thus $\Tria_{D(B \otimes
A)}(B \otimes X_1) \subseteq \mathcal {C}$.

Since $X_1$ is self-compact, by Lemma~\ref{delete B}, for any index
set $\Lambda$ and $T_{\lambda} \in \Tria_{D(B \otimes A)}(B \otimes
X_1)$ with $\lambda \in \Lambda$, we have
$$\begin{array}{lll} \Hom _{D(B \otimes A)}(Z_1,\oplus_{\lambda \in \Lambda}T_{\lambda})
& \cong & \Hom_{D(A)}(X_1,\oplus_{\lambda \in \Lambda}T_{\lambda}) \\
& \cong & \oplus_{\lambda \in \Lambda}\Hom_{D(A)}(X_1,T_{\lambda}) \\
& \cong & \oplus_{\lambda \in \Lambda}\Hom _{D(B \otimes
A)}(Z_1,T_{\lambda}).
\end{array}$$ Thus $Z_1$ is self-compact.

Since $X_1$ is self-compact and exceptional and $\End_{D(A)}(X_1)
\cong A_1$ as algebras, by Lemma~\ref{compact-selfcompact}, we have
$$\Hom _{D(B \otimes A)}(Z_1,Z_1[n]) \cong B \otimes \Hom_{D(A)}(X_1,X_1[n])
\cong \left\{\begin{array}{ll} B \otimes A_1, & \mbox{\rm if } n=0; \\
0, & \mbox{\rm otherwise.} \end{array}\right.$$ Thus $Z_1$ is
exceptional and $\End_{D(B \otimes A)} (Z_1) \cong B \otimes A_1$ as
algebras.

{\it Step 3.} Since $X_2$ is compact and $X_1 \in X_2^{\perp}$, by
Lemma~\ref{compact-selfcompact}, we have
$$\Hom_{D(B \otimes A)}(Z_2,Z_1[n]) \cong B \otimes \Hom_{D(A)}(X_2,X_1[n]) = 0$$
for all $n \in \mathbb{Z}$. Thus $Z_1 \in Z_2^{\perp}$.

{\it Step 4.} For any $Z \in Z_1^{\perp} \cap Z_2^{\perp}$, by
Lemma~\ref{delete B}, we have
$$\Hom_{D(A)}(X_i, Z[n]) \cong \Hom_{D(B \otimes A)}(Z_i, Z[n])=0$$ for all $n \in
\mathbb{Z}$ and $i=1,2$. Thus $Z \in X_1^{\perp} \cap X_2^{\perp} =
\{0\}$. Hence $Z_1^{\perp} \cap Z_2^{\perp} = \{0\}$.

{\it Step 5.} By the Steps 1 to 4 above, we have shown that $Z_1$
and $Z_2$ satisfy all conditions in Proposition~\ref{criterion} for
tensor product algebras. Analogous to J{\o}rgensen's construction
(ref. \cite[Theorem 3.3 and Remark 3.4]{J}), we can obtain a
recollement $(D(B \otimes A_1), D(B \otimes A), D(B \otimes A_2),
I^*, I_*=I_!, I^!, J_!, J^!=J^*, J_*)$ such that
$$\begin{array}{ll} & J_! = - \otimes^L_{B \otimes A_2}(B \otimes Y_2),\\
I_*=I_!= - \otimes^L_{B \otimes A_1}(B \otimes Y_1),
& J^!=J^*=\mbox{\rm RHom}_{B \otimes A}(B \otimes Y_2,-),\\
I^!=\mbox{\rm RHom}_{B \otimes A}(B \otimes Y_1,-), & J_*=\mbox{\rm
RHom}_{B \otimes A_2}(\mbox{\rm RHom}_{B \otimes A}(B \otimes Y_2,B
\otimes A),-).
\end{array}$$

By Lemma~\ref{delete B}, we have $I_*=I_! \cong -
\otimes^L_{A_1}Y_1$, $J_! \cong - \otimes^L_{A_2}Y_2$, $I^! \cong
\RHom_A(Y_1 , -)$ and $J^!=J^* \cong \RHom_A(Y_2 , -)$. Since $Y_2$
is compact in $D(A)$, by Lemma~\ref{delete B} and
Lemma~\ref{compact-selfcompact}, we have
$$\begin{array}{ll} J_* & = \RHom_{B \otimes A_2}(\RHom_{B \otimes
A}(B \otimes Y_2, B \otimes A), -) \\ & \cong \RHom_{B \otimes
A_2}(B \otimes \RHom_A(Y_2,A),-) \\ & \cong
\RHom_{A_2}(\RHom_A(Y_2,A),-).
\end{array}$$

Clearly, $I_* \cong - \otimes^L_{A_1}Y_1 \cong \RHom_{A_1}(Y, -)$
has a left adjoint $- \otimes^L_AY \cong - \otimes^L_{B \otimes A}(B
\otimes Y)$. Thus we have $I^* \cong - \otimes^L_{B \otimes A}(B
\otimes Y) \cong - \otimes^L_AY$.
\end{proof}

\section{Recollements on opposite algebras}

\indent\indent In this section, we show that from a standard
recollement of derived categories of algebras we can obtain that of
opposite algebras.

Let $A$ and $B$ be algebras. Denote by $\rep(B,A)$ the full
subcategory of $D(B^{\op} \otimes A)$ consisting of all complexes of
$B$-$A$-bimodules which are perfect when restricted to complexes of
$A$-modules. The following result is folklore:

\begin{lemma} \label{rep} Let $A$ and $B$ be algebras. Then the derived
Hom functor \linebreak $\RHom_A(-,A) : D(B^{\op} \otimes A)
\rightarrow D(A^{\op} \otimes B)$ induces a duality from $\rep(B,A)$
to $\rep(B^{\op},A^{\op})$.
\end{lemma}

\begin{theorem} \label{op} Let $A,A_1$ and $A_2$ be algebras, and
$Y \in D(A^{\op} \otimes A_1)$ and $Y_2 \in D(A_2^{\op} \otimes A)$
define a standard recollement of $D(A)$ relative to $D(A_1)$ and
$D(A_2)$. Then $Y^\star := \RHom_{A_1}(Y,A_1)$ and $Y_2^* :=
\RHom_A(Y_2,A)$ define a standard recollement of $D(A^{\op})$
relative to $D(A_1^{\op})$ and $D(A_2^{\op})$.
\end{theorem}

\begin{proof} {\it Step 1.} Since $(Y_2)_A$ is compact and exceptional and
$\End_{D(A)}(Y_2) \cong A_2$, we have
$$\Hom_{D(A^{\op})}(Y_2^*, Y_2^*[n]) \cong
\Hom_{D(A)}(Y_2, Y_2[n]) = \left\{\begin{array}{ll} A_2, & \mbox{ if
$n=0$;} \\ 0, & \mbox{ otherwise.}
\end{array}\right.$$ Therefore, $_A(Y_2^*)$ is compact and
exceptional and $\End_{D(A^{\op})}(Y_2^*) \cong A_2^{\op}.$

{\it Step 2.} It follows from Theorem~\ref{tensor} by taking
$B=A_1^\op$ that $Y^\star \otimes^L_A Y \cong A_1$ in $D(A_1^\op
\otimes A_1)$. Thus $Y \otimes^L_{A_1} -$ is a full embedding. Hence
$_AY$ is self-compact in $D(A)$ (ref. \cite[Lemma 1.7]{J}).

By Lemma~\ref{rep}, we have $Y \cong Y^{\star\star}$ in $D(A^\op
\otimes A_1)$. Since $_{A_1}Y^\star$ is compact, we have
$$\begin{array}{rcl} \Hom_{D(A^{\op})}(Y, Y[n]) &
\cong & \Hom_{D(A^{\op})}(Y, Y^{\star\star}[n])
\\ & \cong & \Hom_{D(A^{\op})}(Y, \RHom_{A_1^{\op}}(Y^\star,
A_1[n])) \\ & \cong & \Hom_{D(A_1^{\op})}(Y^\star \otimes^L_A Y,
A_1[n]) \\ & \cong & \Hom_{D(A_1^{\op})}(A_1, A_1[n]) \\
& \cong & H^0(A_1[n]) \\ & \cong & \left\{\begin{array}{ll} A_1, &
\mbox{ if  $n=0$;} \\ 0, & \mbox{ otherwise.}
\end{array}\right. \end{array}$$

Thus $_AY$ is exceptional and $\End_{D(A^{\op})}(Y) \cong A_1^{\op}$
as algebras where the isomorphism is given by sending an element
$a_1 \in A_1$ to the right multiplication of $Y$ by $a_1$.

{\it Step 3.} Since $(Y_2)_A$ is compact and $Y^{\star} \in
Y_2^{\perp}$ in $D(A)$, we have
$$\begin{array}{rcl} \Hom_{D(A^{\op})}(Y_2^*,Y) & \cong &
\Hom_{D(A^{\op})}(Y_2^*, \RHom_{A_1^{\op}}(Y^\star,A_1)) \\ & \cong
& \Hom_{D(A_1^{\op})}(Y^{\star} \otimes^L_AY_2^*,A_1) \\ & \cong &
\Hom_{D(A_1^{\op})}(\RHom_A(Y_2, Y^{\star}),A_1) =0.
\end{array}$$ Thus $Y \in Y_2^{* \perp}$ in $D(A^{\op})$.

{\it Step 4.} For any $X \in Y^\perp \cap (Y_2^*)^\perp \subseteq
D(A^{\op})$, since $_A(Y_2^*)$ is compact and $X \in (Y_2^*)^\perp$,
we have $Y_2 \otimes^L_AX \cong Y_2^{**} \otimes^L_AX \cong
\RHom_{A^\op}(Y_2^*,X)=0$. It follows from Theorem~\ref{tensor} by
taking $B=A^\op$ that there exists a triangle $Y_2^*
\otimes^L_{A_2}Y_2 \rightarrow A \rightarrow Y
\otimes^L_{A_1}Y^{\star} \rightarrow$ \; in $D(A^\op \otimes A)$,
further a triangle $Y_2^* \otimes^L_{A_2}Y_2 \otimes^L_A X
\rightarrow X \rightarrow Y \otimes^L_{A_1}Y^{\star} \otimes^L_A X
\rightarrow$ \; in $D(A^{\op})$. Therefore, $X \cong Y
\otimes^L_{A_1}Y^{\star} \otimes^L_A X$ in $D(A^{\op})$.
Furthermore, $\Hom_{D(A^{\op})}(X,X) \cong \Hom_{D(A^{\op})}(Y
\otimes^L_{A_1}Y^{\star} \otimes^L_A X,X) \cong
\Hom_{D(A_1^{\op})}(Y^{\star} \otimes^L_A X,\RHom_{A^\op}(Y,X))=0$,
since $X \in Y^\perp$. Hence, $Y^\perp \cap (Y_2^*)^\perp=\{0\}$.

{\it Step 5.} By the Steps 1 to 4 above, we have shown that $_AY$
and $_AY_2^*$ satisfy all conditions in Proposition~\ref{criterion}
for opposite algebras. Analogous to J{\o}rgensen's construction
(ref. \cite[Theorem 3.3 and Remark 3.4]{J}), we can obtain a
recollement $(D(A_1^{\op}), D(A^{\op}), D(A_2^{\op}), i^*, i_*=i_!,
i^!, j_!, j^!=j^*, j_*)$ such that
$$\begin{array}{ll} & j_!=Y_2^* \otimes^L_{A_2} -,\\
i_*=i_! = Y \otimes^L_{A_1} -,
& j^!=j^* = \RHom_{A^{\op}}(Y_2^*,-),\\
i^!=\RHom_{A^{\op}}(Y,-), & j_* = \RHom_{A_2^{\op}}(Y_2,-).
\end{array}$$
Clearly, $Y^{\star} \otimes^L_A -$ is a left adjoint of $Y
\otimes^L_{A_1} - \cong \RHom_{A_1^\op}(Y^{\star}, -)$. Thus
$Y^{\star}$ and $Y_2^*$ define a standard recollement of
$D(A^{\op})$ relative to $D(A_1^{\op})$ and $D(A_2^{\op})$.
\end{proof}

\begin{remark} {\rm One referee showed me an elegant proof of Theorem~\ref{op}
by using the correspondence between the smashing subcategories of
$D(A)$ and the idempotent ideals of the category $D^c(A)$ of compact
objects of $D(A)$ (ref. \cite{Kr1,NS1}). Here, we just provide a
relatively elementary proof.}
\end{remark}

\section{\large Recollements and smoothness}

\indent\indent In this section, we shall apply the results obtained
in sections 3 and 4 to study the relation between recollements of
derived categories of algebras and smoothness of algebras. For this,
we need to know the relation between recollements of derived
categories of algebras and global dimensions of algebras. Recently,
the following result is proved:

\begin{proposition} \label{gldim} {\rm (Angeleri
H\"{u}gel-K\"{o}nig-Liu-Yang \cite{AKLY})} Let $A_1,A$ and $A_2$ be
algebras, and $D(A)$ admit a recollement relative to $D(A_1)$ and
$D(A_2)$. Then $A$ is of finite global dimension if and only if so
are $A_1$ and $A_2$.
\end{proposition}

Let $A$ be an algebra and $A^e := A^\op \otimes A$ its enveloping
algebra. The {\it Hochschild dimension} $\dim A$ of $A$ is the
projective dimension of $A$ as a left or right $A^e$-module. The
Hochschild dimensions of algebras were studied very early \cite{CE}.
An algebra $A$ is of Hochschild dimension 0 if and only if $A^e$ is
semisimple \cite[Theorem 7.9]{CE}. In case $A$ is finitely
generated, $A$ is of Hochschild dimension 0 if and only if $A$ is
separable \cite[Theorem 7.10]{CE}. The algebras of Hochschild
dimension $\leq 1$ are called {\it quasi-free} or {\it formally
smooth} \cite{CQ,KR}. An algebra $A$ is said to be {\it smooth} if
it has finite Hochschild dimension, i.e., the projective dimension
of $A$ as $A^e$-module is finite (ref. \cite{VDB}), or equivalently,
$A$ is isomorphic to an object in $K^b(\Proj A^e)$, the homotopy
category of bounded complexes of projective $A^e$-modules. It is
well-known that $A$ is smooth if and only if $\gl A^e < \infty$
where $\gl A^e$ denotes the global dimension of the algebra $A^e$.
Indeed, this follows from the lemma below (ref. \cite[Chap. IX,
Proposition 7.5, 7.6]{CE} and \cite[Proposition 2]{ERZ}).

\begin{lemma} \label{global-dim} Let $A$ and $B$ be algebras. Then the following assertions hold true:

{\rm (1)} $\gl A \leq \dim A \leq \gl A^e$,

{\rm (2)} $\gl A \otimes B \leq \gl A + \dim B$.

\end{lemma}

\begin{remark}{\rm \hspace{1cm}\\
\indent (1) {\it Sometimes $\gl A < \infty \Leftrightarrow \gl A^e <
\infty$}: Let $A$ be either a commutative Noetherian algebra over a
perfect field $k$, or a finite-dimensional $k$-algebra such that the
factor algebra $A/J$ of $A$ modulo its Jacobson radical $J$ is
separable. Then $\gl A < \infty$ if and only if $\gl A^e < \infty$
(ref. \cite[Theorem 2.1]{HKR} and \cite[Theorem 16]{A}).

(2) {\it In general $\gl A < \infty \nRightarrow \gl A^e < \infty$}:
Let $A$ be a finite inseparable field extension of an imperfect
field $k$. Then $\gl A = 0$. However, $\gl A^e = \infty$, since $A
\otimes_k A$ is not semisimple (ref. \cite[Page 65, Remark]{A2}).
}\end{remark}

Let the algebras $A$ and $B$ be derived equivalent. Then by
\cite[Proposition 9.1]{R1} and \cite[Theorem 2.1]{R} we know $A^e$
and $B^e$ are derived equivalent. Thus $\gl A^e < \infty$ if and
only if $\gl B^e < \infty$. Hence, $A$ is smooth if and only if so
is $B$, i.e., the smoothness of algebras is invariant under derived
equivalences. More general, we have the following result:

\begin{theorem} \label{smooth} Let $A_1,A$ and $A_2$ be algebras,
and $D(A)$ admit a recollement relative to $D(A_1)$ and $D(A_2)$.
Then $A$ is smooth if and only if so are $A_1$ and $A_2$.
\end{theorem}

\begin{proof} By Proposition~\ref{standard} and Theorem~\ref{tensor}, we have a recollement of
$D(A^\op \otimes A)$ relative to $D(A^\op \otimes A_1)$ and $D(A^\op
\otimes A_2)$. It follows from Proposition~\ref{gldim} that $\gl A^e
< \infty$ if and only if $\gl A^\op \otimes A_i < \infty$ for all
$i=1,2$. By Proposition~\ref{standard}, Theorem~\ref{op} and
Theorem~\ref{tensor}, we have a recollement of $D(A^\op \otimes
A_i)$ relative to $D(A^\op_1 \otimes A_i)$ and $D(A^\op_2 \otimes
A_i)$. It follows from Proposition~\ref{gldim} that for $i=1,2$,
$\gl A^\op \otimes A_i < \infty$ if and only if $\gl A^\op_j \otimes
A_i < \infty$ for all $j=1,2$. Therefore, $\gl A^e < \infty$ if and
only if $\gl A^e_i < \infty$ for all $i = 1,2$, by
Lemma~\ref{global-dim}.
\end{proof}

Theorem~\ref{smooth} can be applied to judge the smoothness of some
algebras or construct some smooth algebras. For instance, when
applied to triangular matrix algebras, we have the following result:

\begin{corollary} \label{triangular matrix} Let $A_1$ and $A_2$ be
algebras, $M$ an $A_2$-$A_1$-bimodule, and $A={\tiny
\left[\begin{array}{cc} A_1&0\\M&A_2 \end{array}\right]}$. Then $A$
is smooth if and only if so are $A_1$ and $A_2$.
\end{corollary}

\begin{proof} Note that $X_1
:= {\tiny \left[\begin{array}{cc} 1_{A_1}&0\\0&0
\end{array}\right]}A$ and $X_2 := {\tiny \left[\begin{array}{cc}
0&0\\0&1_{A_2}
\end{array}\right]}A$ satisfy all conditions in Proposition~\ref{criterion}.
Thus there is a recollement of $D(A)$ relative to $D(A_1)$ and
$D(A_2)$ (ref. \cite[Corollary 15]{K}). Now the corollary follows
from Theorem~\ref{smooth}.
\end{proof}

\begin{remark}{\rm An algebra $A$ is said to be {\it homologically smooth} if $A$
is compact in $D(A^e)$, i.e., $A$ is isomorphic in $D(A^e)$ to an
object in $K^b(\proj A^e)$ (ref. \cite{KS}). Let $A$ be the infinite
Kronecker algebra {\tiny $\left[\begin{array}{cc} k&0\\V&k
\end{array}\right]$}, where $V$ is an infinite-dimensional $k$-vector
space. Choose $X_2$ to be the simple projective $A$-module and $X_1$
the other simple $A$-module. Then $X_1$ and $X_2$ satisfy all
conditions in Proposition~\ref{criterion}. Thus $D(A)$ admits a
recollement relative to $D(k)$ and $D(k)$ (ref. \cite[Example
9]{K}). By Corollary~\ref{triangular matrix}, we know the infinite
Kronecker algebra is smooth. However, it is not homologically
smooth, because the finitely generated projective $A^e$-module
resolution of $A$ would induce a finitely generated projective
$A$-module resolution of the nonprojective simple $A$-module. Hence,
Corollary~\ref{triangular matrix} and Theorem~\ref{smooth} are not
correct for homological smoothness. }\end{remark}

\section{\large Recollements and Hochschild cohomology}

\indent\indent In this section, we shall apply the results obtained
in sections 3 and 4 to observe the relations between recollements of
derived categories of algebras and Hochschild cohomology of
algebras. Note that the relation between recollements of derived
categories of algebras and Hochschild homology of algebras had been
clarified by Keller in \cite{Ke2}. Recall that the {\it $n$-th
Hochschild homology} of an algebra $A$ is $HH_n(A) :=
\Tor^{A^e}_n(A,A) \cong H^{-n}(A \otimes^L_{A^e}A)$. In $D(k)$ the
complex $A \otimes^L_{A^e}A$ is isomorphic to the Hochschild complex
of $A$. The following result is due to Keller, which is a corollary
of \cite[Theorem 3.1]{Ke2} (ref. \cite[Remarks 3.2 (a)]{Ke2}) and
can be also proved by using Theorem~\ref{tensor}.

\begin{proposition} \label{HH triangle} {\rm (Keller \cite{Ke2})}
Let $A,A_1$ and $A_2$ be algebras, and $D(A)$ admit a recollement
relative to $D(A_1)$ and $D(A_2)$. Then there is a triangle in
$D(k)$:
$$A_2 \otimes^L_{A^e_2} A_2
\rightarrow A \otimes^L_{A^e} A \rightarrow A_1 \otimes^L_{A_1^e}
A_1 \rightarrow .$$
\end{proposition}

From the triangle in Proposition~\ref{HH triangle}, by taking
cohomologies, we can obtain a long exact sequence on the Hochschild
homologies of the algebras:

\begin{corollary} \label{HH long exact sequence} {\rm (Keller \cite{Ke2})}
Let $A,A_1$ and $A_2$ be algebras,
and $D(A)$ admit a recollement relative to $D(A_1)$ and $D(A_2)$.
Then there is a long exact sequence on the Hochschild homologies of
these algebras
$$\cdots \rightarrow HH_{n+1}(A_1) \rightarrow HH_n(A_2) \rightarrow
HH_n(A) \rightarrow HH_n(A_1) \rightarrow \cdots .$$
\end{corollary}

Now we consider Hochschild cohomology. Recall that the {\it $n$-th
Hochschild cohomology} of an algebra $A$ is $HH^n(A) :=
\Ext_{A^e}^n(A,A) \cong H^n(\RHom_{A^e}(A,A))$. Note that in $D(k)$
the complex $\RHom_{A^e}(A,A)$ is isomorphic to the Hochschild
cochain complex or Hochschild cocomplex of $A$. From a recollement
of derived categories of algebras, we shall obtain three triangles
on Hochschild cocomplexes of these algebras, which can induce three
long exact sequences on their Hochschild cohomologies.

The following lemma is essentially due to K\"{o}nig and Nagase (cf.
\cite[Lemma 2.1]{KN}).

\begin{lemma} \label{3 triangles} Let $A$ be an algebra
and $X \stackrel{u}{\rightarrow} Y \stackrel{v}{\rightarrow} Z
\rightarrow$ a triangle in $D(A)$ such that $\RHom_A(X,Z)=0$ in
$D(k)$. Then there are three triangles in $D(k)$:
$$\begin{array}{ll}
(1) & \RHom_A(Y,X) \rightarrow \RHom_A(Y,Y) \stackrel{\phi}{\rightarrow} \RHom_A(Z,Z) \rightarrow \; ,\\
(2) & \RHom_A(Z,Y) \rightarrow \RHom_A(Y,Y) \stackrel{\psi}{\rightarrow} \RHom_A(X,X) \rightarrow \; , \\
(3) & \RHom_A(Z,X) \rightarrow \RHom_A(Y,Y)
\stackrel{\varphi}{\rightarrow} \RHom_A(X,X) \oplus \RHom_A(Z,Z)
\rightarrow .
\end{array}$$
Moreover, $\phi$ (resp. $\psi$, $\varphi$) induces a homomorphism of
graded rings $\bar{\phi}$ (resp. $\bar{\psi}$, $\bar{\varphi}$)
between the corresponding cohomology rings.
\end{lemma}

\begin{proof} Applying the bifunctor $\RHom_A(-,-)$ to the triangle
$X \stackrel{u}{\rightarrow} Y \stackrel{v}{\rightarrow} Z
\rightarrow$, we have the following commutative diagram:

{\scriptsize $$\begin{array}{ccccccc} \RHom_A(X[1],Z[-1]) &
\rightarrow & \RHom_A(X[1],X) & \rightarrow &
\RHom_A(X[1],Y) & \rightarrow & \RHom_A(X[1],Z) \\
\downarrow && \downarrow && \downarrow && \downarrow \\
\RHom_A(Z,Z[-1]) & \rightarrow & \RHom_A(Z,X) & \rightarrow &
\RHom_A(Z,Y) & \rightarrow
& \RHom_A(Z,Z) \\ \downarrow && \downarrow && \downarrow && \downarrow \\
\RHom_A(Y,Z[-1]) & \rightarrow & \RHom_A(Y,X) &
\rightarrow & \RHom_A(Y,Y) & \rightarrow & \RHom_A(Y,Z) \\
\downarrow && \downarrow && \downarrow && \downarrow
\\ \RHom_A(X,Z[-1]) & \rightarrow & \RHom_A(X,X) & \rightarrow &
\RHom_A(X,Y) & \rightarrow & \RHom_A(X,Z)
\end{array}$$}

\noindent in which the four corners are zero by the assumption
$\RHom_A(X,Z)=0$ in $D(k)$. It follows two triangles (1) and (2).

By Octahedral axiom, we have the following commutative diagram:

{\scriptsize $$\begin{array}{ccccccc} & & \RHom_A(Z,Z[-1]) & = & \RHom_A(Z,Z[-1]) & & \\
&& \downarrow && \downarrow && \\
\RHom_A(Z,X) & \rightarrow & \RHom_A(Y,X) & \rightarrow &
\RHom_A(X,X) & \rightarrow &
\RHom_A(Z[-1],X) \\ \| && \downarrow && \downarrow && \| \\
\RHom_A(Z,X) & \rightarrow & \RHom_A(Y,Y) &
\rightarrow & \RHom_A(X,X) \oplus \RHom_A(Z,Z) & \rightarrow & \RHom_A(Z[-1],X) \\
&& \downarrow && \downarrow && \\ & & \RHom_A(Z,Z) & = &
\RHom_A(Z,Z) & &
\end{array}$$}

\noindent where the morphism $\RHom_A(Z,Z[-1]) \rightarrow
\RHom_A(X,X)$ is zero. It follows the triangle (3).

For the last statement, it is enough to note that $\phi$ induces a
map $$\bar{\phi} : \oplus_{n \in \mathbb{Z}} \Hom_{D(A)}(Y,Y[n])
\rightarrow \oplus_{n \in \mathbb{Z}} \Hom_{D(A)}(Z,Z[n])$$ sending
$f_n \in \Hom_{D(A)}(Y, Y[n])$ to the unique morphism
$\bar{\phi}(f_n) \in \Hom_{D(A)}(Z, \linebreak Z[n])$ satisfying
$\bar{\phi}(f_n) \circ v = v[n] \circ f_n$, i.e., the following
diagram in $D(A)$ is commutative:
$$\begin{array}{rcl} Y & \stackrel{v}{\rightarrow} & Z \\
f_n \downarrow && \downarrow \bar{\phi}(f_n) \\
Y[n] & \stackrel{v[n]}{\rightarrow} & Z[n] , \end{array}$$ which is
clearly a homomorphism of graded rings. Similar for $\bar{\psi}$ and
$\bar{\varphi}$.
\end{proof}

The main result in this section is the following:

\begin{theorem} \label{HH^ trianle} Let $A_1,A$ and $A_2$ be algebras,
and $(D(A_1), D(A), D(A_2), i^*,i_* \linebreak =i_!, i^!, j_!,
j^!=j^*, j_*)$ a standard recollement given by $Y \in D(A^{\op}
\otimes A_1)$ and $Y_2 \in D(A_2^{\op} \otimes A)$. Then there are
three triangles in $D(k)$:
$$\begin{array}{ll} (1) &
\RHom_{A^e}(A,\RHom_A(Y_2,A) \otimes^L_{A_2}Y_2) \\ & \rightarrow
\RHom_{A^e}(A,A)
\stackrel{\phi}{\rightarrow} \RHom_{A_1^e}(A_1,A_1) \rightarrow \; , \\
(2) & \RHom_{A^e}(\RHom_{A_1}(Y,Y),A) \\ & \rightarrow
\RHom_{A^e}(A,A) \stackrel{\psi}{\rightarrow} \RHom_{A_2^e}(A_2,A_2) \rightarrow \; , \\
(3) & \RHom_{A^e}(\RHom_{A_1}(Y,Y),\RHom_A(Y_2,A)
\otimes^L_{A_2}Y_2) \\ & \rightarrow \RHom_{A^e}(A,A)
\stackrel{\varphi}{\rightarrow} \RHom_{A_1^e}(A_1,A_1) \oplus
\RHom_{A_2^e}(A_2,A_2) \rightarrow .
\end{array}$$
Moreover, $\phi$ (resp. $\psi$, $\varphi$) induces a homomorphism of
graded rings $\bar{\phi}$ (resp. $\bar{\psi}$, $\bar{\varphi}$)
between the corresponding Hochschild cohomology rings.
\end{theorem}

\begin{proof} By Theorem~\ref{tensor},
we have a recollement $(D(A^{\op} \otimes A_1), D(A^{\op} \otimes
A), \linebreak D(A^{\op} \otimes A_2), I^*, I_*=I_!, I^!, J_!,
J^!=J^*, J_*)$ such that
$$\begin{array}{lll} I^*  \cong - \otimes^L_A Y , & & J_! \cong - \otimes^L_{A_2}Y_2,\\
I_*=I_! \cong \RHom_{A_1}(Y,-), & & J^!=J^* \cong \RHom_A(Y_2,-),\\
I^! \cong \RHom_A(\RHom_{A_1}(Y,A_1),-), & & J_* \cong
\RHom_{A_2}(\RHom_A(Y_2,A),-).
\end{array}$$
Thus we obtain a triangle $J_!J^!A \rightarrow A \rightarrow I_*I^*A
\rightarrow $ in $D(A^{\op} \otimes A)$. Note that
$\RHom_{A^e}(J_!J^!A,I_*I^*A)=0$ due to the recollement. By
Lemma~\ref{3 triangles}, we have three triangles in $D(k)$:
$$\begin{array}{ll}
(1) & \RHom_{A^e}(A,J_!J^!A) \rightarrow \RHom_{A^e}(A,A) \rightarrow \RHom_{A^e}(I_*I^*A,I_*I^*A) \rightarrow \; , \\
(2) & \RHom_{A^e}(I_*I^*A,A) \rightarrow \RHom_{A^e}(A,A)
\rightarrow \RHom_{A^e}(J_!J^!A,J_!J^!A) \rightarrow \; , \\
(3) & \RHom_{A^e}(I_*I^*A,J_!J^!A) \rightarrow \RHom_{A^e}(A,A)
\rightarrow \RHom_{A^e}(J_!J^!A,J_!J^!A) \oplus \\ &
\RHom_{A^e}(I_*I^*A,I_*I^*A) \rightarrow .
\end{array}$$
By Lemma~\ref{rep} and Theorem~\ref{tensor}, we have
$$\begin{array}{lll} \RHom_{A^e}(I_*I^*A,I_*I^*A) & \cong & \RHom_{A^{\op} \otimes A_1}(I^*A,I^*A)\\
& \cong & \RHom_{A^{\op} \otimes A_1}(Y,Y)\\
& \cong & \RHom_{A_1^{\op} \otimes
A}(\RHom_{A_1}(Y,A_1),\RHom_{A_1}(Y,A_1))\\ & \cong &
\RHom_{A_1^e}(A_1,A_1)
\end{array}$$
and
$$\begin{array}{lll} \RHom_{A^e}(J_!J^!A,J_!J^!A) & \cong & \RHom_{A^{\op} \otimes A_2}(J^!A,J^!A)\\
& \cong & \RHom_{A^{\op} \otimes A_2}(\RHom_A(Y_2,A),\RHom_A(Y_2,A))\\
& \cong & \RHom_{A_2^{\op} \otimes A}(Y_2,Y_2)\\ & \cong &
\RHom_{A_2^e}(A_2,A_2),
\end{array}$$
where the last steps follow from the full embeddings $I_*$ and $J_!$
in Theorem~\ref{tensor} by taking $B=A_1^\op$ and $A_2^\op$
respectively. Thus there are three triangles in $D(k)$ as required.

The last statement follows from the isomorphisms above and
Lemma~\ref{3 triangles}.
\end{proof}

From the three triangles in Theorem~\ref{HH^ trianle}, by taking
cohomologies, we can obtain three long exact sequences on the
Hochschild cohomologies of the algebras:

\begin{corollary} \label{HH^ long exact sequences} Let $A_1,A$ and $A_2$ be algebras,
and $(D(A_1), D(A), D(A_2), i^*,i_* \linebreak =i_!, i^!, j_!,
j^!=j^*, j_*)$ a standard recollement given by $Y \in D(A^{\op}
\otimes A_1)$ and $Y_2 \in D(A_2^{\op} \otimes A)$. Then there are
three long exact sequences:
$$\begin{array}{lrl} (1) & \cdots  \rightarrow & \Hom_{D(A^e)}(A,\RHom_A(Y_2,A) \otimes^L_{A_2}Y_2[n])
\\ & \rightarrow & HH^n(A) \stackrel{\phi_n}{\rightarrow} HH^n(A_1) \rightarrow \cdots \; , \\
(2) & \cdots \rightarrow & \Hom_{D(A^e)}(\RHom_{A_1}(Y,Y),A[n]) \\ &
\rightarrow & HH^n(A) \stackrel{\psi_n}{\rightarrow} HH^n(A_2) \rightarrow \cdots \; , \\
(3) & \cdots \rightarrow &
\Hom_{D(A^e)}(\RHom_{A_1}(Y,Y),\RHom_A(Y_2,A) \otimes^L_{A_2}Y_2[n]) \\
& \rightarrow & HH^n(A) \stackrel{\varphi_n}{\rightarrow} HH^n(A_1)
\oplus HH^n(A_2) \rightarrow \cdots.
\end{array}$$
Moreover, $\oplus_{n \in \mathbb{N}} \phi_n$ (resp. $\oplus_{n \in
\mathbb{N}} \psi_n$, $\oplus_{n \in \mathbb{N}} \varphi_n$) is a
homomorphism of graded rings between the corresponding Hochschild
cohomology rings.
\end{corollary}

Applying Corollary~\ref{HH^ long exact sequences} to
Example~\ref{stratifying ideal} by taking $A_1=A/AeA, A_2=eAe, Y_2
\linebreak =eA$ and $Y=A/AeA$, we can obtain the following result
due to K\"{o}nig and Nagase:

\begin{corollary} {\rm (K\"{o}nig-Nagase \cite{KN})}
Let $A$ be an algebra, $e$ an idempotent of $A$ and $AeA$ a
stratifying ideal of $A$. Then there are three long exact sequences:
$$\begin{array}{ll}
(1) & \cdots \rightarrow \Ext^n_{A^e}(A,AeA) \rightarrow HH^n(A)
\stackrel{\psi_n}{\rightarrow} HH^n(A/AeA) \rightarrow \cdots \; , \\
(2) & \cdots \rightarrow \Ext^n_{A^e}(A/AeA,A) \rightarrow
HH^n(A) \stackrel{\phi_n}{\rightarrow} HH^n(eAe) \rightarrow \cdots \; , \\
(3) & \cdots \rightarrow \Ext^n_{A^e}(A/AeA,AeA) \rightarrow HH^n(A)
\stackrel{\varphi_n}{\rightarrow} HH^n(A/AeA) \oplus HH^n(eAe)
\rightarrow \cdots.
\end{array}$$
Moreover, $\oplus_{n \in \mathbb{N}} \phi_n$ (resp. $\oplus_{n \in
\mathbb{N}} \psi_n$, $\oplus_{n \in \mathbb{N}} \varphi_n$) is a
homomorphism of graded rings between the corresponding Hochschild
cohomology rings.
\end{corollary}

\bigskip

\noindent {\footnotesize {\bf ACKNOWLEDGMENT.} I thank Dong Yang for
pointing out an error in the original version and some helpful
discussions. I am indebted to Steffen K\"{o}nig and Dong Yang for
letting me know their recent work \cite{AKLY}. I am grateful to both
referees for their numerous very valuable suggestions which improve
some original results especially the present Theorem~\ref{op} and
Theorem~\ref{smooth} and make the paper much readable. I am
sponsored by Project 10731070 and 11171325 NSFC.}

\end{document}